\documentclass{ifacconf}

\usepackage{graphicx}      
\usepackage{natbib}        

\usepackage{amsmath}
\usepackage{amsfonts}
\usepackage{amssymb}
\usepackage[english]{babel}
\usepackage{enumitem}
\usepackage{mathtools}

\makeatletter
\gdef\th@plain{\upshape
  \def\@begintheorem##1##2{\item[\hskip\labelsep{\theorem@headerfont ##1\ ##2.}]}%
  \def\@opargbegintheorem##1##2##3{\item[\hskip\labelsep{\theorem@headerfont ##1\ ##2.}]{\upshape(##3).}\ \ignorespaces}}
\makeatother

\newtheorem{remark}{Remark}
\begin{document}
\renewcommand{\abstractname}{{\bfseries Abstract: }}
\begin{frontmatter}

\title{On the Stability Analysis of Infinite-Dimensional Port-Hamiltonian Systems Within the Variational Framework} 


\author[First]{Moritz Leitner} 
\author[First]{Lukas Ecker} 
\author[First]{Markus Sch\"oberl}

\address[First]{Institute of Control Systems, Johannes Kepler University Linz, Altenbergerstrasse 69, 4040 Linz, Austria (e-mail: moritz.leitner, lukas.ecker, markus.schoeberl@jku.at).}

\begin{abstract} 
This paper addresses the stability analysis of a class of linear infinite-dimensional port-Hamiltonian systems formulated within the variational framework.
We derive verifiable conditions for the equivalence of the Hamiltonian-induced energy norm and the natural state-space norm, as well as for the generation of a contraction semigroup by the associated system operator.
The semigroup generation result is obtained by establishing dissipativity and verifying the required range condition.
The proposed conditions simplify the functional-analytic stability analysis of this class of systems in a manner similar to established results for the Stokes–Dirac setting.
\end{abstract}


\end{frontmatter}

\section{INTRODUCTION}

The physically-motivated modelling of ordinary differential equations as port-Hamiltonian systems provides a well-established framework for energy-based controller and observer design, see e.g. \cite{ortega}.
Several formulations have been proposed to extend the port-Hamiltonian framework to infinite-dimensional systems, among which the Stokes–Dirac approach as in \cite{stokesdirac1}, \cite{stokesdirac2} and the variational framework as in \cite{jet1}, \cite{jet3} are two prominent examples.
Both approaches have been successfully applied to the design of in-domain and boundary controllers for a variety of systems of partial differential equations (PDEs) with one dimensional spatial domains, see e.g. \cite{Regstokes1}, \cite{Regjet1}, \cite{Regjet2}.

The main difference between the Stokes–Dirac and variational formulations lies in their underlying geometric structures. While the Stokes–Dirac formulation is based on Stokes–Dirac structures, the variational framework is closely related to the calculus of variations and relies on jet-bundle structures. As a consequence, the choice of state variables and the resulting structure of the PDEs differ between the two formulations. In the Stokes–Dirac setting, the state variables are typically chosen such that the energy density can be expressed as a quadratic function of the states. In contrast, within the variational framework, the energy density may depend on both the states and their spatial derivatives. For mechanical systems, the Stokes–Dirac approach commonly uses strains, curvatures, and momenta as state variables, whereas the variational framework allows the direct use of displacements and the corresponding momenta.
This can be advantageous for position-based control schemes such as those considered in \cite{Regjet1}, \cite{Regjet2}.
However, due to the different PDE structures arising from the two formulations, the corresponding stability analysis also differs.

For linear systems in the Stokes–Dirac setting, semigroup generation and stability can be characterized by verifiable matrix-based conditions as stated in \cite{zwart_linear_nodate}. In contrast, for infinite-dimensional port-Hamiltonian systems within the variational framework (ipH-V systems), stability proofs have so far mainly been developed for specific examples see e.g. \cite{malzer_energy-based_2021}, \cite{Ecker}, \cite{malzer_stability_2021}, while general results for this class of systems remain limited.

The goal of this paper is to reduce the complex stability analysis of a general class of ipH-V systems to a set of simply verifiable conditions, thereby greatly reducing the required functional-analytic effort.
For linear ipH-V systems, the proof of stability requires showing that the associated system operator generates a contraction semigroup, which can be established by the application of the Lumer-Phillips theorem see e.g. \cite{zwart_linear_nodate}.
In order to do so, we make use of an energy norm induced by the Hamiltonian of the system, where we have to show that it constitutes a norm of the underlying state space.
We show how under certain assumptions, which naturally occur for physical systems, these steps can be reduced to an examination of the boundary conditions.

The remainder of this paper is organized as follows.
Section \ref{sec:theory} introduces the notational, differential-geometric and functional-analytic tools used throughout the paper.
In Section \ref{sec:sys}, the considered class of ipH-V systems is formulated in a differential-geometric setting and subsequently represented in a form suitable for functional-analytic analysis.
Next, Section \ref{sec:main} presents the main result of this contribution together with a detailed proof.
Finally, Section \ref{sec:example} illustrates the proposed framework using a boundary controlled Timoshenko beam.

\section{THEORETICAL BACKGROUND}\label{sec:theory}

\subsection{Notation}

In this contribution, we consider linear distributed-parameter systems defined on a one-dimensional spatial domain with spatial coordinate $\xi \in [0,L]$.
We will make use of the standard vector-valued Hilbert spaces $L^2([0, L]; \mathbb{R}^n)$, as well as the Sobolev spaces $W^{1, 2}([0, L]; \mathbb{R}^n)$, which will be abbreviated as $L^2$ and $W^{1, 2}$. 
The standard inner products and norms will be denoted as $\bigr<., .\bigl>_{L^2}$, $\bigr<., .\bigl>_{W^{1, 2}}$, $||.||_{L^2}$, and $||.
||_{W^{1, 2}}$ respectively, see e.g. \cite{brezis_functional_2011}.
Furthermore, $||.||$ denotes the Euclidean norm of a vector, $I$ denotes the identity matrix, $\sigma_{\max}(.)$ denotes the largest singular value of a matrix and $\mathcal{L}(\mathcal{A}; \mathcal{B})$ denotes the set of all linear operators mapping from $\mathcal{A}$ to $\mathcal{B}$.

\subsection{Differential Geometric Preliminaries}

The upcoming system description relies on the use of jet manifolds, which contain the independent spatial coordinate $\xi$, $\dim(z) = n$ dependent variables $z$ as well as the derivatives of $z$ with respect to $\xi$ denoted as jet-variables.
In our case only the first and second order jet-variables $z_\xi$ and $z_{\xi\xi}$ are required.
Furthermore, the abbreviations $\partial_\xi = \partial /\partial \xi$, $\partial_z = \partial/\partial z$ and $\partial_z^\xi = \partial / \partial z_\xi$ will be used.
Given a map $z = \phi(\xi)$ that relates the dependent to the independent coordinates we can state
$
    z_\xi \circ \phi = \partial_\xi \phi(\xi)$ and $ z_{\xi\xi} \circ \phi = \partial_\xi  (\partial_\xi \phi(\xi))$.
This allows us to introduce the first $\mathcal{J}^1(\mathcal{E})$ and second $\mathcal{J}^2(\mathcal{E})$ jet manifold with coordinates $(\xi, z, z_\xi)$ and $(\xi, z, z_\xi, z_{\xi\xi})$ respectively, where $\mathcal{E}$ is equipped with the coordinates $(\xi, z)$.
Throughout this contribution, we will deal with evolution equations of the form
 \begin{equation} \label{eq:PDE}
    \dot{z} = v\left(\xi, z, z_\xi, z_{\xi\xi}\right)
 \end{equation}
together with suitable boundary conditions.
Furthermore, we are interested in densities $\mathfrak{F} = \mathcal{F} \mathrm{d}\xi$, with $\mathcal{F} \in C'(\mathcal{J}^1(\mathcal{E}))$ that can be integrated over the spatial domain $F = \int_0^L\mathcal{F} \mathrm{d}\xi$.
\footnote{Here $C'(\mathcal{J}^1(\mathcal{E}))$ denotes the set of functions on the manifold $\mathcal{J}^1(\mathcal{E})$ that are sufficiently often differentiable.}
The change of $F$ along solutions of \eqref{eq:PDE}, assuming they exist, can now be expressed as
    $\dot{F} = \int_0^L (\partial_z \mathcal{F})v + (\partial_z^\xi \mathcal{F})(d_\xi v) \mathrm{d} \xi \, ,$
with the total derivative $d_\xi = \partial_\xi + z_\xi\partial_z + z_{\xi\xi} \partial_z^\xi$.
Integration by parts together with Stokes's theorem results in
 \begin{equation} \label{eq:lie}
    \dot{F} = \int_0^L \underbrace{\left(\partial_z \mathcal{F} - d_\xi\left(\partial_z^\xi \mathcal{F}\right)\right)}_{\delta \mathfrak{F}}v \mathrm{d} \xi + (\underbrace{(\partial_z^\xi \mathcal{F})}_{\delta^\partial \mathfrak{F}}v) \bigr|_0^L \, ,
\vspace{-2 pt}
 \end{equation}
with the variational derivative $\delta$ and the boundary operator $\delta^\partial$, see e.g. \cite{jet3} for a precise differential geometric treatment of this topic.
Accordingly, the second term on the right hand side of \eqref{eq:lie} is denoted as a boundary port.
This relation is a key result that will allow us to study PDEs as ipH-V systems in the upcoming sections.

\subsection{Functional Analytic Preliminaries}

Next, we will discuss some functional analytic identities that will be frequently used in the following derivations.
Given $g, h \in L^2$, if
 \begin{equation} \label{eq:weak}
    \int_0^L \partial_\xi v^\mathrm{T} g \mathrm{d \xi} = - \int_0^L v^\mathrm{T} h \mathrm{d} \xi
 \end{equation}
holds for every infinitely differentiable $v: [0, L] \rightarrow\mathbb{R}^n$ with $v(0) = v(L) = 0$, then $h$ is the weak derivative of $g$, denoted as $h = D_\xi g$ (\cite{brezis_functional_2011}).
Considering $K = K^\mathrm{T}: [0, L] \rightarrow \mathbb{R}^{n \times n}$, the assumptions $\gamma_1 I \leq K$ and $0< \gamma_1$ allow us to state
 \begin{equation} \label{eq:lower}
        \gamma_1||g||^2_{L^2} \leq \int_0^L g^\mathrm{T} K g \mathrm{d}\xi \, ,
 \end{equation}
while, under the assumptions $0 \leq K \leq \gamma_2 I$ and $\gamma_2 < \infty$, 
 \begin{equation} \label{eq:upper}
        \biggl|\int_0^L g^\mathrm{T} K h \mathrm{d}\xi\biggr| \leq \gamma_2 ||g||_{L^2}||h||_{L^2}
 \end{equation}
holds (\cite{zwart_linear_nodate}).
Furthermore, we can state a similar result for a general, not necessarily symmetric matrix $B:[0, L] \rightarrow \mathbb{R}^{n \times n}$
 \begin{equation} \label{eq:upperB}
    \biggl|\int_0^L g^\mathrm{T} B h \mathrm{d}\xi\biggr| \leq \operatorname{ess}\sup \left(\sigma_{\max}(B)\right) ||g||_{L^2}||h||_{L^2} \, .
 \end{equation}
Additionally, for $g \in W^{1, 2}$, we will frequently use the so called trace inequalities
 \begin{equation} \label{eq:trace}
    ||g(0)|| \leq \gamma_0 ||g||_{W^{1, 2}}\, , \quad ||g(L)|| \leq \gamma_L ||g||_{W^{1, 2}} \, ,
 \end{equation}
as in \cite{grisvard_elliptic_1985}.
Next, we want to recall two well-known theorems related to differential operators.
Let us consider linear evolution equations of the form $\dot{x} = \mathcal{A}x$ with suitable boundary conditions and $x \in \mathcal{X}$ denoting the state space.
\begin{thm}[Lumer-Phillips Theorem \cite{zwart_linear_nodate}, Theorem 6.1.7]\label{th:lumer}
    Let $\mathcal{A}$ be a linear operator with domain $\mathcal{D}(\mathcal{A)}$ on a Hilbert space $\mathcal{X}$.
    Then $\mathcal{A}$ is the infinitesimal generator of a contraction semigroup on $\mathcal{X}$ if and only if $\mathcal{A}$ is dissipative and $\mathrm{ran}\left(I - \mathcal{A}\right) = \mathcal{X}$.
\end{thm}
\begin{thm}[\cite{pazy_semigroups_nodate}, Theorem 4.5] \label{th:pazy}
    Let $\mathcal{A}$ be a dissipative operator in $\mathcal{X}$.
    If for some $\lambda_0 > 0$, $\mathrm{ran}\left(\lambda_0 I - \mathcal{A}\right) = \mathcal{X}$ then $\mathrm{ran}\left(\lambda I - \mathcal{A}\right) = \mathcal{X}$ for all $\lambda > 0$. 
\end{thm}
To conclude this section, two more well known results that will be used in the subsequent analysis are stated.
\begin{thm}[Lax-Milgram Lemma \cite{grisvard_elliptic_1985}, Lemma 2.2.1.1]
    Let $a$ be a bounded and coercive bilinear form on a Hilbert space $\mathcal{Q}$.
    For every bounded linear functional $F$ on $\mathcal{Q}$, there exists a unique solution $q \in \mathcal{Q}$ such that
    $
        a(q, v) = F(v)
    $
    for every $v \in \mathcal{Q}$.
\end{thm}
\begin{thm}[Peetre-Tartar Lemma \cite{ern_finite_2021}, Lemma A.20] \label{th:peetre}
Let $\mathcal{Q}$, $\mathcal{Y}$, $\mathcal{Z}$ be three Banach spaces.
Let $\mathcal{U} \in \mathcal{L}(\mathcal{Q}; \mathcal{Y})$ be an injective operator and let $\mathcal{V}\in \mathcal{L}(\mathcal{Q}; \mathcal{Z})$ be a compact operator.
Assume there is $\alpha > 0$ such that $\alpha ||q||_\mathcal{Q} \leq ||\mathcal{U}(q)||_\mathcal{Y} + ||\mathcal{V}(q)||_\mathcal{Z}$ for all $q \in \mathcal{Q}$.
Then there is $\beta > 0$ such that
$
    \beta ||q||_\mathcal{Q} \leq ||\mathcal{U}(q)||_\mathcal{Y}, \quad \forall q \in \mathcal{Q}
$.
\end{thm}

\section{System Representation}\label{sec:sys}

A special case of the autonomous evolution equations \eqref{eq:PDE} are ipH-V systems, which can be written as
 \begin{equation} \label{eq:iphd}
    \dot{x} = \left(\mathcal{J} - \mathcal{R}\right)(\delta \mathfrak{H})^\mathrm{T}\, ,
 \end{equation}
with the skew-symmetric interconnection tensor $\mathcal{J}$, the symmetric positive semidefinite dissipation map $\mathcal{R}$, the Hamiltonian density $\mathfrak{H} = \mathcal{H}\mathrm{d}\xi$ and the variational derivative $\delta$ as well as a set of suitable Robin and Dirichlet boundary conditions.
As the evolution equations \eqref{eq:PDE} contain spatial derivatives up to the order two we restrict ourselves to $\mathcal{H} \in C'(\mathcal{J}^1(\mathcal{E}))$.
Furthermore, we deal with the canonical case, which describes systems of $n$ hyperbolic PDEs stemming from an Euler-Lagrange formalism, extended by dissipative effects, reformulated to systems of $2n$ PDEs by a Legendre transformation, see e.g. \cite{jet1}.
Therefore, the state can be separated into $\dim(q) = n$ displacements and $\dim(p) = n$ corresponding momenta, with $p = M \dot{q}$.
Thus, $x = \begin{bmatrix}
    q^\mathrm{T} & p^\mathrm{T}
\end{bmatrix}^\mathrm{T}$ and
 \begin{equation*}
    \mathcal{J} = \begin{bmatrix}
        0 & I \\
        -I & 0
    \end{bmatrix} , \quad \mathcal{R} = \begin{bmatrix}
        0 & 0 \\
        0 & \mathcal{R}_{22}
    \end{bmatrix} .
 \end{equation*}
As in this case, the Hamiltonian density is dependent on the deflections $q$, their spatial derivatives $q_\xi$ and the momenta $p$, we introduce the state space as $\mathcal{X} = \mathcal{Q} \times L^2$ accordingly.
Here, similarly to \cite{kuchment_quantum_2004},
 \begin{equation} \label{eq:ass1}
    \mathcal{Q} = \{q \in W^{1, 2} \bigr| (I-P) q_\partial = 0\} \, , \ \quad q_\partial = \begin{bmatrix}
    q\bigr|_L \\
    q\bigr|_0
\end{bmatrix} ,
 \end{equation}
contains the Dirichlet boundary conditions with the orthogonal projector
 \begin{equation} \label{eq:ass2} \tag{A1}
    P=P^\mathrm{T}=P^2 \, .
 \end{equation}
It is important to note, that $I - P$ is itself an orthogonal projector.
The inner product and norm are given by $\bigl<x, \bar{x}\bigr>_\mathcal{X} = \bigl<q, \bar{q}\bigr>_{W^{1, 2}} + \bigl<p, \bar{p}\bigr>_{L^2}$ and $||x||^2_\mathcal{X} = \bigl< x, x\bigr>_\mathcal{X}$ respectively.
Based on $\mathcal{X}$ we can introduce the bilinear forms
 \begin{equation*}
    \bar{\mathcal{H}}_q(q, q_\xi, \bar{q}, \bar{q}_\xi) = \frac{1}{2}\begin{bmatrix}
        q^\mathrm{T} & q_\xi^\mathrm{T}
    \end{bmatrix}
    \underbrace{\begin{bmatrix}
        A & B \\
        B^\mathrm{T} & C
    \end{bmatrix}}_S
    \begin{bmatrix}
        \bar{q} \\
        \bar{q}_\xi
    \end{bmatrix}
 \end{equation*}
and 
 \begin{equation*}
    \bar{\mathcal{H}}_p(p, \bar{p}) = \frac{1}{2} p^\mathrm{T} M^{-1} \bar{p} \, .
 \end{equation*}
The matrix $B: [0, L] \rightarrow \mathbb{R}^{n \times n}$ and the symmetric matrices $A,C, M, \mathcal{R}_{22}: [0, L] \rightarrow \mathbb{R}^{n \times n}$ satisfy  the assumptions
\begin{gather}
    0 \leq A,C, M, \mathcal{R}_{22} \leq \gamma_s I \tag{A2} \label{eq:ass3} \\
    \operatorname{ess}\sup \left(\sigma_{\max}(B)\right) \leq \gamma_s \tag{A3} \label{eq:ass4} \\
    0 \leq A - B C^{-1} B^\mathrm{T} \tag{A4} \label{eq:ass5} \\
    \gamma_i I \leq C, M \tag{A5} \label{eq:ass6} 
\end{gather}
with some finite constants $\gamma_s < \infty$ and $0 < \gamma_i$.
\footnote{Requiring also $\gamma_i I \leq A$ would be too restrictive for many physical systems.}
Together these conditions also ensure that $S$ is positive semidefinite $S\geq 0$.
Building upon these constructions, candidates for the energy inner product and norm are given by
 \begin{equation} \label{eq:energiIP}
    \bigr<x, \bar{x}\bigl>_\mathcal{H} = \bigl<q, \bar{q}\bigl>_{\mathcal{H}_q} + \bigl<q, \bar{q}\bigl>_{H_B} + \bigl<p, \bar{p}\bigl>_{\mathcal{H}_p}
 \end{equation}
and
 \begin{equation} \label{eq:H}
    ||x||^2_\mathcal{H} = \bigr<x, x\bigl>_\mathcal{H} = \underbrace{\bigl<q, q\bigl>_{\mathcal{H}_q}}_{||q||^2_{\mathcal{H}_q}} + \underbrace{\bigl<q, q\bigl>_{H_B}}_{||q||^2_{H_B}} + \underbrace{\bigl<p, p\bigl>_{\mathcal{H}_p}}_{||p||^2_{\mathcal{H}_p}} \, ,
 \end{equation}
respectively, with $\bigl<q, \bar{q}\bigl>_{\mathcal{H}_q} = \int_0^L \bar{\mathcal{H}}_q(q, q_\xi, \bar{q}, \bar{q}_\xi) \mathrm{d}\xi$, $\bigl<p, \bar{p}\bigl>_{\mathcal{H}_p} = \int_0^L \bar{\mathcal{H}}_p(p, \bar{p}) \mathrm{d}\xi$ and $\bigl<q, \bar{q}\bigl>_{H_B} = \frac{1}{2} q_\partial^\mathrm{T} T \bar{q}_\partial$ satisfying
 \begin{equation} \label{eq:ass7}
    T = T^\mathrm{T} \geq 0\, . \tag{A6}
 \end{equation}
It is important to note that, in the case of a rank deficient $A$, $\bigl<q, \bar{q}\bigl>_{\mathcal{H}_q}$ might not constitute an inner product under certain boundary conditions, which motivates the inclusion of $\bigl<q, \bar{q}\bigl>_{H_B}$.
\footnote{This addition will also allow us to include boundary proportional controllers in the proposed system description.}
To ensure that the energy inner product turns $\mathcal{X}$ into a Hilbert space, we establish the equivalence of the induced energy norm and the natural state-space norm. Specifically, we seek positive constants $\bar{\gamma}_i$ and $\bar{\gamma}_s$ such that
 \begin{equation} \label{eq:ineq}
    \bar{\gamma}_i||x||^2_\mathcal{X} \leq ||x||^2_\mathcal{H} \leq \bar{\gamma}_s||x||^2_\mathcal{X}
 \end{equation}
holds.
\begin{thm}\label{th:normeq}
    Let $||x||^2_\mathcal{H}$ be a bilinear form according to \eqref{eq:H}.
    With $\mathcal{Q}$ as in \eqref{eq:ass1} and the assumptions \eqref{eq:ass2}-\eqref{eq:ass7} as well as
     \begin{equation} \label{eq:cond}
        \{q \in \mathcal{Q} : ||q||^2_{\mathcal{H}_q} + ||q||^2_{H_B} = 0 \} = \{0\} \, , \tag{A7}
     \end{equation}
    ensuring that only the trivial kernel can exist, equation \eqref{eq:ineq} can be satisfied and $||x||^2_\mathcal{H}$ is equivalent to the natural norm on the state space $||x||^2_\mathcal{X}$.
\end{thm}

For the proof we refer to the Appendix.

In this paper we will consider Hamiltonians of the special form $\mathcal{H} = \bar{\mathcal{H}}_q(q, q_\xi, q, q_\xi) + \bar{\mathcal{H}}_p(p, p)$ thus, the ipH-V system \eqref{eq:iphd} corresponds to a set of linear PDEs.
Furthermore, the energy norm can be rewritten as
$
    ||x||^2_\mathcal{H} = \int_0^L \mathcal{H} \mathrm{d} \xi + ||q||^2_{H_B}\, .
$
Next, we want to study how $||x||^2_\mathcal{H}$ changes along solutions of \eqref{eq:iphd}.
We can use \eqref{eq:lie} and replace $v$ with \eqref{eq:iphd} and $\mathfrak{F}$ with $\mathfrak{H}$.
This results in
$
        \frac{d}{dt}||x||^2_\mathcal{H} = -\int_0^L (\delta \mathfrak{H})\mathcal{R}(\delta \mathfrak{H})^\mathrm{T} + ((\delta^\partial_q \mathfrak{H})\dot{q} )\bigr|_0^L + q_\partial^\mathrm{T} T \dot{q}_\partial
$
which states that the change of energy in the system is determined by dissipation within the spatial domain and power exchange through the boundary.
Using the fact that $\mathcal{R}_{22}$ is positive semidefinite (see \eqref{eq:ass3}), we obtain
 \begin{equation} \label{eq:power}
     \frac{d}{dt}||x||^2_\mathcal{H} \leq \left(\begin{bmatrix}
        \delta_q^{\partial} \mathfrak{H}\bigr|_L &
        -\delta_q^{\partial} \mathfrak{H}\bigr|_0
    \end{bmatrix} + q_\partial^\mathrm{T} T\right)\dot{q}_\partial \, .
 \end{equation}
For the considered Hamiltonian density, the boundary contributions arise exclusively from the spatial derivatives of the generalized displacements $q$. Following the terminology in \cite{zwart_linear_nodate}, we introduce the boundary effort vector $e_\partial$ as
 \begin{equation*}
    e_\partial = \begin{bmatrix}
        \delta_q^{\partial} \mathfrak{H}\bigr|_L &
        -\delta_q^{\partial} \mathfrak{H}\bigr|_0
    \end{bmatrix}^\mathrm{T} =  \begin{bmatrix}
        \left(B^\mathrm{T}q + C q_\xi \right)\bigr|_L \\
        -\left(B^\mathrm{T}q + C q_\xi \right)\bigr|_0
    \end{bmatrix}.
 \end{equation*}
For the subsequent analysis, we impose a set of $2n$ boundary conditions combining Robin and Dirichlet conditions as
 \begin{equation} \label{eq:WB}
    \begin{bmatrix}
        P & I-P
    \end{bmatrix} \begin{bmatrix}
        e_\partial + T q_\partial + K \dot{q}_\partial \\
        q_\partial
    \end{bmatrix} = 0 \, ,
 \end{equation}
with
 \begin{equation} \label{eq:ass9}
    K = K^\mathrm{T} \geq  0 \tag{A8}
 \end{equation}
and $P$ as in \eqref{eq:ass1}.
The complementary orthogonal projectors $P$ and $I-P$ decompose the boundary space into subspaces subject to Robin and Dirichlet boundary conditions, respectively.
Under these assumptions, we can rewrite the power balance \eqref{eq:power} as
 \begin{equation*}
    \begin{aligned}
        &\frac{d}{dt}||x||^2_\mathcal{H} \leq (e^\mathrm{T}_\partial + q_\partial^\mathrm{T} T) \dot{q}_\partial \\
        & \qquad = (e^\mathrm{T}_\partial + q_\partial^\mathrm{T}T + \dot{q}_\partial^\mathrm{T} K) \dot{q}_\partial -\dot{q}_\partial^\mathrm{T} K \dot{q}_\partial = -\dot{q}_\partial^\mathrm{T} K \dot{q}_\partial \leq 0\, ,
    \end{aligned}
 \end{equation*}
where we used the boundary conditions to eliminate the indefinite term.
Furthermore, $\dot{q}_\partial$ can be expressed as
 \begin{equation*}
    \dot{q}_\partial = (M^{-1}p)_\partial = \begin{bmatrix}
        M^{-1}p \bigr|_L \\
        M^{-1}p \bigr|_0
    \end{bmatrix} \, .
 \end{equation*}

To apply the functional-analytic results introduced above, we represent the system of PDEs as a differential operator.
Accordingly, we can write \eqref{eq:iphd} as
 \begin{equation} \label{eq:operator}
    \dot{x} = \mathcal{A}x = \left(\mathcal{J} - \mathcal{R}\right)(\delta \mathfrak{H})^\mathrm{T}
 \end{equation}
with the domain $\mathcal{D}\left(\mathcal{A}\right)$ of the operator
 \begin{equation*}
    \begin{aligned}
        \mathcal{D}\left(\mathcal{A}\right) :=
        &\{x \in \mathcal{X} \bigr| M^{-1}p \in \mathcal{Q}, w \in W^{1, 2}, \\ 
        & \quad \begin{bmatrix}
        P & I-P
    \end{bmatrix} \begin{bmatrix}
        e_\partial + T q_\partial + K (M^{-1}p)_\partial \\
        q_\partial
    \end{bmatrix} = 0\} \,,
    \end{aligned}
 \end{equation*}
where $w = B^\mathrm{T} q + C q_\xi$ denotes the flux associated with the PDE system.
This physically motivated construction of $w$ and $\mathcal{D}\left(\mathcal{A}\right)$ also allows for systems with discontinuous $q_\xi$.
To conclude this section, we relate the energy balance to the dissipativity of the system operator.
Using \eqref{eq:operator} together with the energy inner product \eqref{eq:energiIP} we obtain
 \begin{equation} \label{eq:dissi}
\begin{aligned}
    \frac{d}{dt}||x||^2_\mathcal{H} &= \bigl<\mathcal{A}x, x\bigr>_\mathcal{H} + \bigl<x, \mathcal{A}x\bigr>_\mathcal{H} \\
    &= 2\mathrm{Re} \bigl< \mathcal{A}x, x \bigr>_\mathcal{H} \leq 0 \, ,
\end{aligned}
 \end{equation}
which immediately results in the operator $\mathcal{A}$ being dissipative on the Hilbert space $\mathcal{X}$ with norm $||x||_\mathcal{H}$.
\begin{remark}
    The proposed system class already covers several common control structures. Boundary damping injection and proportional boundary control can be incorporated through a suitable choice of $T$, $K$ and $P$.
    Additionally, the system class under investigation can also cover the analysis of the error system regarding in-domain distributed damping and proportional control, which can be represented through modifications of $A$ or $\mathcal{R}_{22}$. 
\end{remark}

\section{Main Theorem}\label{sec:main}

Based on the system representation and the functional-analytic results established in the preceding sections, we now state the main theorem, which provides sufficient conditions for the generation of a contraction semigroup.
\begin{thm} \label{th:main}
    Let $\mathcal{A}$ be a linear operator associated with an ipH-V system as defined in \eqref{eq:operator} with domain $\mathcal{D}(\mathcal{A})$ and  $\mathcal{Q}$ as in \eqref{eq:ass1}.
    Under the assumptions \eqref{eq:ass2}-\eqref{eq:ass9}, $\mathcal{A}$ is the infinitesimal generator of a contraction semigroup on the Hilbert space $\mathcal{X}$ with norm $||x||_\mathcal{H}$ if the boundary conditions can be written in the form \eqref{eq:WB}.
\end{thm}
\begin{pf}
In order to apply Theorem \ref{th:lumer} we need to show the dissipativity of $\mathcal{A}$ as well as the range condition $\mathrm{ran}\left(I - \mathcal{A}\right) = \mathcal{X}$.
Dissipativity of $\mathcal{A}$ follows directly from \eqref{eq:dissi}.
According to \cite{liu_semigroups_nodate}, $\mathrm{ran}\left(\lambda_0 I - \mathcal{A}\right) = \mathcal{X}$ holds for a sufficiently small $\lambda_0 > 0$ if $\mathcal{A}^{-1}$ exists and is bounded.
Thus, using Theorem \ref{th:pazy} implies that $\mathrm{ran}\left(\lambda I - \mathcal{A}\right) = \mathcal{X}$ also holds for $\lambda =1$.
To establish the invertibility of $\mathcal{A}$, we consider the operator equation
 \begin{equation} \label{eq:inv}
        \mathcal{A} \begin{bmatrix}
            q \\
            p
        \end{bmatrix} = \begin{bmatrix}
            r \\
            s
        \end{bmatrix}, 
 \end{equation}
for arbitrary $\begin{bmatrix}
            r^\mathrm{T} & s^\mathrm{T}
        \end{bmatrix}^\mathrm{T} \in \mathcal{X}$ and seek a unique solution $\begin{bmatrix}
            q^\mathrm{T} & p^\mathrm{T}
        \end{bmatrix}^\mathrm{T} \in \mathcal{D}(\mathcal{A})$.
The expanded evolution equations \eqref{eq:operator}
 \begin{equation*} \label{eq:extended}
        \begin{bmatrix}
                    M^{-1}p \\
                    d_\xi \left(B^\mathrm{T} q + C q_\xi\right) - B q_\xi - Aq - \mathcal{R}_{22} M^{-1}p
                \end{bmatrix} = \begin{bmatrix}
            r \\
            s
        \end{bmatrix}
 \end{equation*}
can be solved for $p$, as $M  \geq \gamma_i I$ resulting in $p = Mr$.
The result can be substituted into the second line, yielding
$
        \frac{1}{2}\left(-d_\xi \left(B^\mathrm{T} q + C q_\xi\right) + B q_\xi + Aq\right) = -\frac{1}{2}\left(s + \mathcal{R}_{22}r\right)=g$.
According to \cite{liu_semigroups_nodate} the weak solution of this equation is sufficient to prove the invertibility.
Therefore, the equation is multiplied by a test function $v \in \mathcal{Q}$ and integrated over the whole domain
        $\frac{1}{2}\int_0^L -v^\mathrm{T} d_\xi \left(B^\mathrm{T} q + C q_\xi\right) + v^\mathrm{T}B q_\xi + v^\mathrm{T} Aq \mathrm{d} \xi = \int_0^L v^\mathrm{T}g \mathrm{d} \xi$.
Integration by parts results in
 \begin{equation} \label{eq:aqv}
    \begin{aligned}
        \frac{1}{2}\int_0^L \partial_\xi v^\mathrm{T} C q_\xi + \partial_\xi v^\mathrm{T} B^\mathrm{T} q + v^\mathrm{T} B q_\xi + v^\mathrm{T} A q \mathrm{d} \xi\\ - \frac{1}{2} v_\partial^\mathrm{T}e_\partial = \int_0^L v^\mathrm{T}g \mathrm{d} \xi \,,
    \end{aligned}
 \end{equation}
with $v^\mathrm{T}_\partial = \begin{bmatrix}
    v^\mathrm{T}\bigr|_L & v^\mathrm{T}\bigr|_0
\end{bmatrix}$.
To incorporate the boundary contributions into the variational formulation, \eqref{eq:WB} is used. Since the test functions satisfy $v \in \mathcal{Q}$, we have $(I-P)v_\partial=0$, and thus
 \begin{equation*}
    \begin{aligned}
        v_\partial^\mathrm{T} e_\partial &= v_\partial^\mathrm{T} (I-P) e_\partial + v_\partial^\mathrm{T} P e_\partial \\
        & = v_\partial^\mathrm{T} P (- T q_\partial - K \dot{q}_\partial) = - v_\partial^\mathrm{T} T q_\partial - v_\partial^\mathrm{T} K \dot{q}_\partial \, .
    \end{aligned}
 \end{equation*}
Substituting $\dot{q} = M^{-1}p$ and $p = Mr$ yields $v_\partial^\mathrm{T} e_\partial =- v_\partial^\mathrm{T} T q_\partial - v_\partial^\mathrm{T} K r_\partial$, with $r_\partial = \begin{bmatrix}
    r^\mathrm{T}\bigr|_L & r^\mathrm{T} \bigr|_0
\end{bmatrix}^\mathrm{T}$ and therefore the boundary term can be split up into a contribution to $a(q, v)$ and one to $F(v)$.
Thus
 \begin{equation*}
    \begin{aligned}
        \underbrace{\frac{1}{2}\int_0^L \partial_\xi v^\mathrm{T} C q_\xi + \partial_\xi v^\mathrm{T} B^\mathrm{T} q + v^\mathrm{T} B q_\xi + v^\mathrm{T} A q \mathrm{d} \xi + \frac{1}{2}v^\mathrm{T}_\partial T q_\partial}_{a\left(q, v\right)} \\
        = \underbrace{\int_0^L v^\mathrm{T}g \mathrm{d} \xi -\frac{1}{2} v_\partial^\mathrm{T} K r_\partial}_{F\left(v\right)} \, ,
    \end{aligned}
 \end{equation*}
allows for the application of the Lax-Milgram lemma to conclude on the existence of a unique solution $q \in \mathcal{Q}$.
Consequently, we need to show that both $a(q, v)$ and $F(v)$ are bounded and $a(q, q)$ is coercive.
To do so, we will use the relations $a(q, v) = \bigl<q, v\bigr>_{\mathcal{H}_q} + \bigl<q, v\bigr>_{H_B}$ and $a(q, q) = ||q||_{\mathcal{H}_q}^2 + ||q||^2_{H_B}$.
As already shown in the proof of Theorem \ref{th:normeq}, both $||q||_{\mathcal{H}_q}^2$ and $||q||^2_{H_B}$ are bounded.
Thus, the Cauchy-Schwarz inequality yields the boundedness of $a(q, v)$ by
        $|a(q, v)| \leq ||q||_{\mathcal{H}_q}||v||_{\mathcal{H}_q} + ||q||_{H_B}||v||_{H_B} \leq \bar{\gamma}_s ||q||_{W^{1, 2}}||v||_{W^{1, 2}}$.
Similarly, coercivity follows from \eqref{eq:ineq} as
 \begin{equation} \label{eq:abound}
    \bar{\gamma}_i ||q||_{W^{1, 2}}^2 \leq ||q||_{\mathcal{H}_q}^2 + ||q||^2_{H_B} = a(q, q) \, .
 \end{equation}
Using the trace inequalities \eqref{eq:trace} together with \eqref{eq:upper} and \eqref{eq:ass3} results in the boundedness of $F(v)$ as
 \begin{equation} \label{eq:Fbound}
    \begin{aligned}
        |F\left(v\right)| &\leq |\int_0^L -v^\mathrm{T} \frac{1}{2}\left( s + \mathcal{R}_{22} r\right) \mathrm{d} \xi - \frac{1}{2}v_\partial^\mathrm{T} K r_\partial| \\
        & \leq \left( \frac{1}{2}||s||_{L^2} + \gamma_r||r||_{W^{1, 2}}\right)||v||_{W^{1, 2}}\, ,
    \end{aligned}
 \end{equation}
with $\gamma_r = \gamma_s + \frac{1}{2}((\gamma_0)^2 + (\gamma_L)^2)\sigma_{\max}(K)$.
Thus, the Lax-Milgram lemma can be applied, yielding a unique $q$ that fulfills \eqref{eq:inv}.
It remains to show $\begin{bmatrix}
    q^\mathrm{T} & p^\mathrm{T}
\end{bmatrix}^\mathrm{T} \in \mathcal{D}(\mathcal{A)}$.
Since $r \in \mathcal{Q}$ also $M^{-1}p = r \in \mathcal{Q}$.
From the definition of $w$ together with the triangle inequality, \eqref{eq:upper}, \eqref{eq:upperB}, \eqref{eq:ass3} and \eqref{eq:ass4} we can state
$
        ||w||_{L^2} = ||B^\mathrm{T} q + C q_\xi||_{L^2}
        \leq \gamma_s\left(||q||_{L^2} + ||q_\xi||_{L^2}\right) \leq \gamma_s ||q||_{W^{1, 2}}
$.
As $q \in \mathcal{Q}$ the right hand side and therefore also $||w||_{L^2}$ is bounded.
Acknowledging that in the definition of the weak derivative \eqref{eq:weak}, the test functions vanish at the boundary allows us to rewrite \eqref{eq:aqv} as
 \begin{equation*}
    \int_0^L \partial_\xi v^\mathrm{T} w \mathrm{d} \xi = - \int_0^L v^\mathrm{T} \underbrace{(Aq + B q_\xi - 2g)}_h \mathrm{d} \xi \, .
 \end{equation*}
Comparing this to \eqref{eq:weak} shows that $h$ is the weak derivative of $w$.
We have already shown $g \in L^2$.
Furthermore, by applying the triangle inequality, \eqref{eq:upper}, \eqref{eq:upperB}, \eqref{eq:ass3} and \eqref{eq:ass4},
    $||Aq + B q_\xi||_{L^2} \leq \gamma_s\left(||q||_{L^2} + ||q_\xi||_{L^2}\right) \leq \gamma_s ||q||_{W^{1, 2}}$.
Thus $||h||_{L^2}$ is bounded and $w \in W^{1, 2}$.
The Dirichlet boundary conditions are trivially fulfilled from $q \in \mathcal{Q}$, whereas the Robin boundary conditions require additional calculations.
Reordering $D_\xi w = h = Aq + Bq_\xi - 2g$ yields $\tfrac{1}{2}(-D_\xi w + Bq_\xi + Aq) = g$.
Additionally, using integration by parts, we can state
$
    \frac{1}{2}\int_0^L \partial_\xi v ^\mathrm{T}w \mathrm{d} \xi = -\frac{1}{2}\int_0^L v^\mathrm{T} D_\xi w \mathrm{d}\xi + \frac{1}{2} v_\partial^\mathrm{T} e_\partial
$.
Substituting both of these results into $a(q, v) = F(v)$ the interior terms cancel out and we are left with
 \begin{equation*}
    \begin{aligned}
        &v_\partial^\mathrm{T}(e_\partial + Tq_\partial + Kr_\partial) \\
        &\qquad = v_\partial^\mathrm{T}(I-P+P)(e_\partial + Tq_\partial + Kr_\partial) = 0 \, ,
    \end{aligned}
 \end{equation*}
which, together with $v \in \mathcal{Q}$, results in the Robin boundary conditions of the form \eqref{eq:WB}.
Thus, $\mathcal{A}^{-1}$ exists.
Finally boundedness of $\mathcal{A}^{-1}$ can be derived from $||p||_{L^2} = ||Mr||_{L^2} \leq \gamma_s ||r||_{L^2}$ as well as the combination of \eqref{eq:abound} and \eqref{eq:Fbound}, resulting in
$
        \bar{\gamma}_i ||q||_{W^{1, 2}}^2 \leq a(q, q) = |F(q)|
        \leq \left( \frac{1}{2}||s||_{L^2} + \gamma_r||r||_{W^{1, 2}}\right)||q||_{W^{1, 2}}
$
and thus $||q||_{W^{1, 2}} \leq \tfrac{1}{2\bar{\gamma}_i}||s||_{L^2} + \tfrac{\gamma_r}{\bar{\gamma}_i}||r||_{W^{1, 2}}$. \hfill$\blacksquare$
\end{pf}

\section{EXAMPLE}\label{sec:example}

In this section, we deal with a practical example, namely a Timoshenko beam that is clamped to a guided roller at $\xi = 0$ and free at $\xi = L$.
The PDEs of such a beam are given by
 \begin{equation*}
    \begin{aligned}
        \rho \ddot{w} &= d_\xi\left(kG \left(w_\xi - \psi\right)\right)\\
        J \ddot{\psi} &= kG \left(w_\xi - \psi\right) + d_\xi \left(EI \psi_\xi\right)
    \end{aligned}
 \end{equation*}
which are subject to the boundary conditions \\ $-kG\left(w_\xi - \psi\right)\bigr|_0 =u$, $kG\left(w_\xi - \psi\right)\bigr|_L =0$, $\psi\bigr|_0 = 0$ and $EI \psi_\xi\bigr|_L = 0$,
where $w$ denotes the vertical deflection of the beam, $\psi$ describes the rotation of the cross section relative to the normal of the beam middle axis and $u$ is a boundary force that can act as a control input.
Within the variational framework, the states are $w$ and $\psi$, together with the corresponding momenta $p_w = \rho \dot{w}$, $p_\psi = J \dot{\psi}$.
\footnote{Contrary to this, the Stokes-Dirac scenario uses the shear deformation $\epsilon_t = w_\xi - \psi$ and the curvature $\epsilon_r = \psi_\xi$ together with the same momenta as the variational framework as states.}
We now want to derive a proportional boundary control law that stabilizes an equilibrium $w_s$, $\psi_s = 0$.
In order to do so we can introduce the error states $\tilde{w} = w - w_s$, $\tilde{\psi} = \psi-\psi_s$, $p_{\tilde{w}} = \rho \dot{\tilde{w}}$ and $p_{\tilde{\psi}} = J \dot{\tilde{\psi}}$
which allow us to formulate a Hamiltonian $\tilde{\mathcal{H}} = \tilde{\mathcal{H}}_q + \tilde{\mathcal{H}}_p$ with
 \begin{equation*}
    \tilde{\mathcal{H}}_q = \frac{1}{2}\begin{bmatrix}
        \tilde{w} \\
        \tilde{\psi} \\
        \tilde{w}_\xi \\
        \tilde{\psi}_\xi
    \end{bmatrix}^\mathrm{T} \begin{bmatrix}
        0 & 0 & 0 & 0 \\
        0 & kG & -kG & 0  \\
        0 & -kG & kG & 0 \\
        0 & 0 & 0 & EI 
    \end{bmatrix}
    \begin{bmatrix}
        \tilde{w} \\
        \tilde{\psi} \\
        \tilde{w}_\xi \\
        \tilde{\psi}_\xi
    \end{bmatrix}
 \end{equation*}
and
 \begin{equation*}
    \tilde{\mathcal{H}}_p = \frac{1}{2}\begin{bmatrix}
        p_{\tilde{w}} \\
        p_{\tilde{\psi}}
    \end{bmatrix}^\mathrm{T} \begin{bmatrix}
        \frac{1}{\rho} & 0 \\
        0 & \frac{1}{J} \\
    \end{bmatrix}
    \begin{bmatrix}
        p_{\tilde{w}} \\
        p_{\tilde{\psi}}
    \end{bmatrix}
 \end{equation*}
such that the error system of the beam relative to the equilibrium can be written in the desired port-Hamiltonian manner
 \begin{equation} \label{eq:Timoshenkoiphd}
    \underbrace{\begin{bmatrix}
        \dot{\tilde{w}} \\
        \dot{\tilde{\psi}} \\
        \dot{p}_{\tilde{w}} \\
        \dot{p}_{\tilde{\psi}}
    \end{bmatrix}}_{\dot{\tilde{x}}} = \underbrace{\begin{bmatrix}
        0 & 0 & 1 & 0 \\
        0 & 0 & 0 & 1 \\
        -1 & 0 & 0 & 0 \\
        0 & -1 & 0 & 0
    \end{bmatrix} \begin{bmatrix}
        \delta_{\tilde{w}} \tilde{\mathcal{H}} \\
        \delta_{\tilde{\psi}} \tilde{\mathcal{H}} \\
        \delta_{p_{\tilde{w}}} \tilde{\mathcal{H}} \\
        \delta_{p_{\tilde{\psi}}} \tilde{\mathcal{H}} \\
    \end{bmatrix}}_{\tilde{\mathcal{A}}\tilde{x}} \, .
 \end{equation}
The norm for the subsequent stability analysis is
$
    ||\tilde{x}||^2_{\tilde{\mathcal{H}}} = \int_0^L \tilde{\mathcal{H}} \mathrm{d} \xi + ||\tilde{q}||^2_{\tilde{H}_B}
$
and contains the boundary term $||\tilde{q}||^2_{\tilde{H}_B} = \tfrac{1}{2}\tilde{q}^\mathrm{T}_\partial T \tilde{q}_\partial$ with $\tilde{q}^\mathrm{T}_\partial = \begin{bmatrix}
        \tilde{w} \bigr|_L & \tilde{\psi} \bigr|_L & \tilde{w} \bigr|_0 & \tilde{\psi} \bigr|_0
    \end{bmatrix}$ and $T = \text{diag}(0, 0, 1, 0)$.
With $P = \text{diag}(1, 1, 1, 0)$, $K = \text{diag}(0, 0, 1, 0)$ and the boundary efforts taking the form
 \begin{equation*}
    \tilde{e}_\partial = \begin{bmatrix}
        kG \left(\tilde{w}_\xi - \tilde{\psi}\right)\bigr|_L \\
        EI \tilde{\psi}_\xi\bigr|_L \\
        -kG \left(\tilde{w}_\xi - \tilde{\psi}\right)\bigr|_0 \\
        -EI \tilde{\psi}_\xi\bigr|_0
    \end{bmatrix}
 \end{equation*}
the Robin boundary conditions can be stated according to \eqref{eq:WB} as
 \begin{equation} \label{eq:timorobin}
    \underbrace{\begin{bmatrix}
        1 & 0 & 0 & 0 \\
        0 & 1 & 0 & 0 \\
        0 & 0 & 1 & 0 \\
        0 & 0 & 0 & 0
    \end{bmatrix}}_P \left( \vphantom{\begin{bmatrix} 0 \\ 0 \\ 0 \\ 0 \end{bmatrix}}
    \tilde{e}_\partial + \smash{\underbrace{\begin{bmatrix}
        0 \\
        0 \\
        \tilde{w} \bigr|_0 \\
        0
    \end{bmatrix}}_{T \tilde{q}_\partial} + \underbrace{\begin{bmatrix}
        0 \\
        0 \\
        \dot{\tilde{w}} \bigr|_0 \\
        0
    \end{bmatrix}}_{K \dot{\tilde{q}}_\partial}} \right) = 0 \, ,
 \end{equation}
while the Dirichlet boundary conditions take the form
 \begin{equation}\label{eq:timodirich}
    \underbrace{\begin{bmatrix}
        0 & 0 & 0 & 0 \\
        0 & 0 & 0 & 0 \\
        0 & 0 & 0 & 0 \\
        0 & 0 & 0 & 1
    \end{bmatrix}}_{I-P} \underbrace{\begin{bmatrix}
        \tilde{w} \bigr|_L \\
        \tilde{\psi} \bigr|_L \\
        \tilde{w} \bigr|_0 \\
        \tilde{\psi} \bigr|_0
    \end{bmatrix}}_{\tilde{q}_\partial} = 0 \, .
 \end{equation}
This results in the control law $u = - (w|_0 - w_s) - (\dot{w}|_0 - \dot{w}_s)$.
Thus, we can calculate the change of $||\tilde{x}||^2_{\tilde{\mathcal{H}}}$ to
$
    \frac{d}{dt}||\tilde{x}||^2_{\tilde{\mathcal{H}}} =  (\tilde{e}^\mathrm{T}_\partial + \tilde{q}^\mathrm{T}_\partial T) \dot{\tilde{q}}_\partial = -(\dot{w}|_0 - \dot{w}_s)^2 \leq 0
$
and conclude that $\tilde{\mathcal{A}}$ is dissipative.
Furthermore, the Dirichlet boundary conditions \eqref{eq:timodirich} together with the boundary proportional control ensure the validity of \eqref{eq:cond}.
As a result, we can state that \eqref{eq:Timoshenkoiphd} together with the boundary conditions \eqref{eq:timorobin} and \eqref{eq:timodirich} is indeed the infinitesimal generator of a contraction semigroup according to Theorem \ref{th:main} and therefore the error system is stable.

\section{CONCLUSIONS}

We were able to show that, although the variational formalism differs greatly from the Stokes-Dirac formalism, similar, easily verifiable conditions for the generation of a contraction semigroup by the system operator can be derived.
This marks a first step towards bridging the gap between the differential geometric system description and functional analytic methods, which are required for several proofs related to such ipH-V systems.
One advantage of the approach using the Peetre-Tartar and Lax-Milgram theorems is, that it can be extended towards systems with two dimensional spatial domain or second order Hamiltonian, which will be the topic of future research.





\bibliography{ifacconf.bib}             
                                                   







\appendix
\section{Proof of Theorem \ref{th:normeq}}
\begin{pf}
    To show the equivalence of the energy norm and the natural state space norm we can split up $||x||^2_\mathcal{H}$ into a part depending on $p$ and a part depending on $q$.
    Thus we seek positive constants $\bar{\gamma}_i$ and $\bar{\gamma}_s$ such that $\bar{\gamma}_i ||p||^2_{L^2} \leq ||p||^2_{\mathcal{H}_p} \leq \bar{\gamma}_s ||p||^2_{L^2}$ and $\bar{\gamma}_i ||q||^2_{W^{1,2}} \leq ||q||^2_{\mathcal{H}_q} + ||q||^2_{H_B}\leq \bar{\gamma}_s ||q||^2_{W^{1,2}}$ hold.
    Examining $||p||^2_{\mathcal{H}_p}$ an upper and lower bound exist according to \eqref{eq:lower} and \eqref{eq:upper} together with \eqref{eq:ass3} and \eqref{eq:ass6}.
    Similarly, an upper bound for $||q||^2_{H_B}$ can be derived using $\sigma_{\max}(T)$ together with the trace inequalities \eqref{eq:trace}.
    Applying \eqref{eq:upper} and \eqref{eq:upperB} together with \eqref{eq:ass3} and \eqref{eq:ass4} allows to conclude on an upper bound for $||q||^2_{\mathcal{H}_q}$.
    In order to show the existence of a lower bound for $||q||^2_{\mathcal{H}_q} + ||q||^2_{H_B}$ we will apply Theorem \ref{th:peetre}.
    Therefore, we first introduce the Hilbert space $\mathcal{Q}$ according to \eqref{eq:ass1} as well as the Banach spaces $\mathcal{Y} = L^2 \times \mathbb{R}^{2n}$ and $\mathcal{Z} = L^2$ together with their corresponding norms $||.||^2_\mathcal{Q} = ||.||^2_{W^{1, 2}}$, $||\begin{bmatrix}
        a & b
    \end{bmatrix}^\mathrm{T}||^2_\mathcal{Y} = ||a||^2_{L^2} + ||b||^2$
    and $||.||^2_\mathcal{Z} = ||.||^2_{L^2}$.
    Based on \eqref{eq:ass1}, we can introduce the map $\tau (q) = (I-P)q_\partial$ and declare $\mathcal{Q}$ as the nullspace of this map $\mathcal{Q} = \ker (\tau)$.
    Furthermore, by using the trace inequalities \eqref{eq:trace} the continuity of $\tau$ can be shown as
$
        ||\tau(q) - \tau(\bar{q})|| = ||(I- P)(q_\partial - \bar{q}_\partial)||
        \leq \sigma_{\max}(I-P) (\gamma_0+\gamma_L) ||q - \bar{q}||_{W^{1, 2}}
$.
    As a result, $\mathcal{Q}$ is a closed subspace of the Hilbert space $W^{1, 2}$ and therefore also a Banach space (\cite{brezis_functional_2011}).
    Next, we can introduce the operators $\mathcal{U}^\mathrm{T}(q) = \tfrac{1}{\sqrt2}\begin{bmatrix}\begin{bmatrix}
                q^\mathrm{T} &
                q_\xi^\mathrm{T}
            \end{bmatrix}
            (S^{\frac{1}{2}})^\mathrm{T}  &
            q_\partial^\mathrm{T} (T^{\frac{1}{2}})^\mathrm{T}
        \end{bmatrix}$ and $\mathcal{V}(q) = q$,
    with $(S^{\frac{1}{2}})^\mathrm{T}S^{\frac{1}{2}} = S$ and $(T^{\frac{1}{2}})^\mathrm{T}T^{\frac{1}{2}} = T$, such that their corresponding norms are $||\mathcal{U}(q)||^2_\mathcal{Y} = ||q||^2_{\mathcal{H}_q} + ||q||^2_{H_B}$ and $||\mathcal{V}(q)||^2_\mathcal{Z} = ||q||^2_{L^2}$ respectively.
    To show the injectivity of $\mathcal{U}(q)$ we can use \eqref{eq:cond}.
    Since $\mathcal{U}(q) = 0$ implies $||q||^2_{\mathcal{H}_q} + ||q||^2_{H_B} = 0$, $q=0$ follows from \eqref{eq:cond}, therefore ensuring the injectivity of $\mathcal{U}(q)$.
    The compactness of $\mathcal{V}(q)$ follows as $\mathcal{Q}$ compactly embeds in $\mathcal{Z}$ (\cite{brezis_functional_2011}).
    Lastly, to apply Theorem \ref{th:peetre} we need to show that there exists $\alpha > 0$, such that
    $
        \alpha ||q||_{W^{1, 2}}^2 \leq ||q||^2_{\mathcal{H}_q} + ||q||^2_{H_B} + ||q||^2_{L^2}
    $
    holds.
    The addition of $||q||^2_{L^2}$ can be reformulated to calculating $||q||_{\mathcal{H}_q}^2$ with an extended matrix $\bar{A} = A + 2I$.
    From $0 \leq A - B C^{-1} B^\mathrm{T}$, it is clear that $2I \leq \bar{A} - B C^{-1} B^\mathrm{T}$, which, together with \eqref{eq:ass6}, ensures the uniform positive definiteness of $\bar{S}$ and therefore the existence of a lower bound $\alpha I \leq \bar{S}$ with $\alpha > 0$.
    Thus, Theorem \ref{th:peetre} allows us to conclude that there exists $\beta > 0$ such that
    $
        \beta ||q||^2_{W^{1, 2}} \leq ||q||^2_{\mathcal{H}_q} + ||q||^2_{H_B}
    $
    holds. \hfill$\blacksquare$
\end{pf}
\end{document}